\documentclass[12pt]{amsart}   % For Latex2e
\usepackage[english]{babel}
\usepackage{mathtools}
\usepackage{amssymb,amscd, graphics,color,latexsym,cancel}   % For Latex2e
\usepackage[linktoc=all, pagebackref, hyperindex]{hyperref}%
\hypersetup{colorlinks,  citecolor=blue,  filecolor=blue,  linkcolor=blue,  urlcolor=black}
\usepackage{tikz-cd} 
\usepackage{extpfeil}

\usepackage{tikz}
\usetikzlibrary{matrix,calc}
\usepackage{mathrsfs}
\usepackage[all]{xy}

\usepackage{amsfonts}
\usepackage[normalem]{ulem}

\usepackage{euscript}
\usepackage{amsmath}
\usepackage{ifpdf}
\usepackage{cleveref}
\LARGE\textwidth=6in
\newcommand{\rar}{\rightarrow}
\newcommand{\lar}{\longrightarrow}
\newcommand{\surjects}{\twoheadrightarrow}

\newtheorem{Theorem}{Theorem}[section]
\newtheorem{Lemma}[Theorem]{Lemma}
\newtheorem{Corollary}[Theorem]{Corollary}
\newtheorem{Proposition}[Theorem]{Proposition}
\newtheorem{claim}{}[section]
\theoremstyle{definition}
\newtheorem{Remark}[Theorem]{Remark}
\newtheorem{Example}[Theorem]{Example}

\newtheorem{Definition}[Theorem]{Definition}
\newtheorem{Question}[Theorem]{Question}
\def\sqr#1#2{{\vcenter{\hrule height.#2pt
			\hbox{\vrule width.#2pt height#1pt \kern#1pt
				\vrule width.#2pt}
			\hrule height.#2pt}}}
\def\phi{\varphi}

\def\VaVa{{\mathcal V}\kern-5pt {\mathcal V}}
\def\gr#1#2{{\rm gr}\, _{#1}(#2)}
\def\gr{{\rm gr}\,}
\def\hht{{\rm ht}\,}
\def\depth{{\rm depth}\,}

\def\Min{{\rm Min}\,}
\def\codim{{\rm codim}\,}
\def\spec{{\rm Spec}\,}

\def\ker{{\rm ker}\,}
\def\grade{{\rm grade}\,}
\def\rk{\rm rank}

\def\sym#1#2{\mbox{\rm Sym}_{#1}(#2)}

\def\Ext#1#2#3#4{{\rm Ext}\,^{#1}_{#2}({#3},{#4})}
\def\spec#1{{\rm Spec}\, (#1)}

\def\supp#1{{\rm Supp}\, (#1)}

\def\ini{\mbox{\rm in}}

\def\Rees{{\mathcal R}}
\def\sym{{\mathrm{Sym}}}
\def\cl#1{{\mathcal #1}}

\def\Rees{\mathcal{R}}

\def\phi{\varphi}
\def\Rees{{\cal R}}
\def\hht{{\rm ht}\,}
\def\grade{{\rm grade}\,}
\def\rt{{\rm rt}\,}

\def\xx{{\bf x}}

\def\fa{{\mathfrak a}}

\def\fm{{\mathfrak m}}
\def\fn{{\mathfrak n}}

\def\fp{{\mathfrak p}}

\def\fm{{\mathfrak m}}
\def\fn{{\mathfrak n}}

\def\cl#1{{\cal #1}}
\def\rk{\rm rank}

\newcommand{\excise}[1]{}
 \def\NZQ{\mathbb}               % the font for N,Z,Q,R,C

 \def\CC{{\NZQ C}}
 \def\AA{{\NZQ A}}
 \def\PP{{\NZQ P}}

 \def\Jc{{\mathcal J}}
 
 \def\G{{\mathcal G}}

\def\Q{{\mathcal Q}}
 \def\opn#1#2{\def#1{\operatorname{#2}}} % to make operators
 \opn\chara{char} \opn\length{\ell} \opn\pd{pd} \opn\rk{rk}
 \opn\projdim{proj\,dim} \opn\injdim{inj\,dim} \opn\rank{rank}
 \opn\depth{depth} \opn\grade{grade} \opn\height{height}
 \opn\embdim{emb\,dim} \opn\codim{codim}
 \def\OO{{\mathcal O}}
 \opn\Tr{Tr} \opn\bigrank{big\,rank}
 \opn\superheight{superheight}\opn\lcm{lcm}
 \opn\trdeg{tr\,deg}%\emph{
 \opn\reg{reg} \opn\lreg{lreg} \opn\ini{in} \opn\lpd{lpd}
 \opn\size{size} \opn\sdepth{sdepth}
 \opn\link{link}\opn\fdepth{fdepth}\opn\lex{lex}
 \opn\tr{tr}
 \opn\type{type}
 \opn\div{div} \opn\Div{Div} \opn\cl{cl} \opn\Cl{Cl}
 \opn\Spec{Spec} \opn\Supp{Supp} \opn\supp{supp} \opn\Sing{Sing}
 \opn\Ass{Ass} \opn\Min{Min}\opn\Mon{Mon}
 \opn\Ho{H}
 \opn\Ann{Ann} \opn\Rad{Rad} \opn\Soc{Soc}
 \opn\Im{Im} \opn\Ker{Ker} \opn\Coker{Coker} \opn\Am{Am}
 \opn\Hom{Hom} \opn\Tor{Tor} \opn\Ext{Ext} \opn\End{End}
 \opn\Aut{Aut} \opn\id{id}
 
 \opn\nat{nat}
 \opn\pff{pf}%   \pf exists already
 \opn\Pf{Pf} \opn\GL{GL} \opn\SL{SL} \opn\mod{mod} \opn\ord{ord}
 \opn\Gin{Gin} \opn\Hilb{Hilb}\opn\sort{sort}
 \opn\PF{PF}\opn\Ap{Ap}
 \opn\aff{aff} \opn
 \con{conv} \opn\relint{relint} \opn\st{st}
 \opn\lk{lk} \opn\cn{cn} \opn\core{core} \opn\vol{vol}  \opn\inp{inp} \opn\nilpot{nilpot}
 \opn\link{link} \opn\star{star}\opn\lex{lex}\opn\set{set}
 \opn\width{wd}
 \opn\Fr{F}
 \opn\QF{QF}
 \opn\G{G}
 \opn\type{type}\opn\res{res}
 \opn\log{Log}
 \opn\gr{gr}
 \def\Rees{{\mathcal R}}
 \def\pot#1#2{#1[\kern-0.28ex[#2]\kern-0.28ex]}

 \opn\dirlim{\underrightarrow{\lim}}
 \opn\inivlim{\underleftarrow{\lim}}
\begin{document}
	%\begin{titlepage}

\title[the relation type of varieties]{the relation type of varieties}
\author{Maryam Akhavin and  Abbas Nasrollah Nejad}
\address{Department of Mathematics, Institute for Advanced Studies in Basic Sciences (IASBS), Zanjan 45137-66731,
	Iran}
\email{maryam.akhavin@gmail.com, abbasnn@iasbs.ac.ir}

\begin{abstract}
In this paper, we introduce the notion of relation type of analytic and formal algebras and prove that it is well-defined and  invariant by describing this notion in terms of the Andr\'e-Quillen homology and using the Jacobi-Zariski long exact sequence of homology. In particular, the relation type is an invariant of schemes of finite type over a field, analytic varieties, and algebroid varieties.  
\end{abstract}

\maketitle
\section*{introduction}
Let $(X,\OO_X)$ be either a scheme of finite type over a field $k$, a complex analytic variety, or an algebroid $k$-variety. We introduce the \textit{relation type of} $X$ to be the maximum of the relation type of the local ring $\OO_{X,x}$ of $X$ at the point $x\in X$, where $\OO_{X,x}$ is either a finitely generated $k$-algebra, an analytic $\mathbb{C}$-algebra, or a formal $k$-algebra, respectively. We show that the relation type of $X$  is an invariant in the categories of schemes of finite type over a field $k$,  complex analytic varieties, or algebroid $k$-varieties. 

A \textit{standard graded $R$-algebra} is a commutative graded algebra $A=\bigoplus_{n\geq 0}A_n$ with $A_0=R$ and $A$ is generated as an $R$-algebra by  some elements of degree $1$. Assume that $A_1$ is generated, as an $R$-module, by the elements $f_1,\ldots, f_m$.  Let $S=R[T_1,\ldots,T_m]$ be a polynomial ring over $R$ with variables $T_1,\ldots,T_m$  and let $\varphi\colon S\surjects A$ be the \textit{polynomial presentation} of $A$ sending $T_i$ to $f_i$. Let $\Q=\bigoplus_{\ell\geq 1}\Q_{\ell}$ stand for the kernel of $\varphi$ whose generators will be referred as the \textit{defining equation}  or \textit{defining ideal} of $A$. 

Let $\Q\left<\ell\right>$ stand for the ideal generated by elements in $\Q$ of degree at most $\ell$.   The \textit{relation type}  of $A$, denoted by $\mathrm{rt}_R(A)$, is the least integer $\ell\geq 1$ such that $\Q\left<\ell\right>=\Q$, i.e., the maximum degree in a minimal generating set of homogeneous elements of $\Q$. Since the isomorphism $S/\Q\cong A$ is graded, an application of the Schanuel Lemma to the graded pieces shows that the notion is independent of the generators of $A_1$. Let $\sym_{R}(A_1)$ stand for the symmetric algebra of the $R$-module $A_1$.  The isomorphism of $R$-algebra $\sym_R(A_1)\cong S/\Q\left<1\right>$ holds and there is a natural $R$-algebra homomorphism $\alpha\colon\sym_R(A_1)\surjects A$, which we call  a \textit{canonical symmetric presentation} of $A$. We denote by $\Jc$ the kernel of $\alpha$ and hence $\Jc=\Q/\Q_1$. 

The relation type of a standard graded $R$-algebra $A$ can be given in terms of the Andr\'e-Quillen homology. This approach was established by F. Planas-Vilanova in~\cite{FV1}. We explain in~\ref{VilanovaApproch} the procedure he employed. In fact, the relation type of $A$ is the minimum positive integer $r\geq 1$ such that the $n$-th graded component of the first Andr\'e-Quillen homology $\Ho_1(R,A,R)_n$ is zero for every $n\geq r+1$.  We also observe that $\rt_R(A)=1$ if and only if the second Andr\'e-Quillen homology $\Ho_2(A,R,R)=0$ (see~\ref{secondAQPresentation}).
The advantage of the definition of the relation type in terms of the  Andr\'e-Quillen homology is that one can prove, using the base change property of the cotangent complex, some algebraic properties such as  the relation type being a local invariant and being stable under the adic-completion which is stated in Proposition~\ref{basechange}. 

Let $R$ be a Noetherian ring and $I$  an ideal of $R$. The two most common and important standard graded  $R$-algebras related to the ideal $I$ are the \textit{Rees algebra} and the \textit{associated graded ring}. Recall that these algebras are defined as 
\[\Rees_{R}(I)=\bigoplus_{n\geq 0}I^tt^n\cong R[It]\subseteq R[t],\,\quad  \gr_{I}(R)=\bigoplus_{n\geq 0}I^n/I^{n+1}. \]

The relation type of $I$, denoted by $\rt_R(I)$, is defined as $\rt_R(I):=\rt_R(\Rees_{R}(I))$. The  ideal $I$ is said to be \textit{of linear type} if $\rt_R(I)=1$.  It is proved by G. Valla that $I$ is of linear type if and only if the canonical symmetric presentation $\beta:\sym_{R/I}(I/I^2)\surjects \gr_{I}(R)$ is an isomorphism~\cite[Theorem 1.3]{Valla}.
This result was generalized by F. Planas-Vilanova~\cite[Proposition 3.3]{FV1} and W. Heinzer, M.K. Kim and B. Ulrich~\cite[Discussion]{HKU}  using the module of effective relations and its relation with the graded Koszul complex and the extended symmetric algebra analogue to the extended Rees algebra, respectively. Namely, for an ideal $I\subseteq R$
\[
\mathrm {rt}_R(I)=\mathrm {rt}_R(\Rees_R(I))=\mathrm{rt}_{R/I}(\gr_{I}(R)).
\] 
Therefore, one can  deduce at once that the top degree equations in  the defining ideal of $\Rees_R(I)$ are in correspondence with the top degree equations of the defining ideal of $\gr_{I}(R)$. In Proposition~\ref{reproved} we reprove this result by using the downgrading homomorphism introduced by J. Herzog, A. Simis and W. Vasconcelos~\cite[Page 471]{H.S.V}. 

The main question emerges as to whether  the relation type of an ideal $I$ is 
invariant of the quotient ring $R/I$. This question is valid for polynomial 
rings over a field~\cite[Theorem 3.1]{FV-RT}. Then the relation type is an 
invariant for affine $k$-algebras and affine algebraic varieties. We prove 
that the invariance question is valid for formal power series rings over a field, 
convergence power series rings over a complete valued field, regular local rings, 
complete  local Noetherian rings and  algebraic power series rings. 
Therefore, the relation type is an invariant  in the categories of  formal $k$-algebras, analytic $k$-algebras ($k$ is a complete valued field) and  the category of  Nash $k$-algebras (see Theorem~\ref{T1},\,  Corollary~\ref{cor-vilanova},\, Corollary~\ref{cor-vilanova2} and Remark~\ref{Nashrings}). In Lemma~\ref{k-iso} we generalize~\cite[Exercise 13, Chapter V, §5]{kunz} to formal and analytic $k$-algebras, which indeed  has a crucial role in  the proof of Theorem~\ref{T1}.

The final section of the paper is devoted to the definition of relation type of varieties. Let $X$ be an affine $k$-scheme of finite type. We define the relation type of $X$ as the relation type of coordinate ring of $X$ which is a finitely generated $k$-algebra. This leads us to define the relation type of a scheme of finite type over $k$ as the maximum of the relation type of the local ring $\OO_{X,x}$ at any point $x\in X$. We prove that any regular $k$-scheme of finite type and any reduced closed subscheme of $\AA_k^2$ is of relation type one ( see Proposition~\ref{regular} and Example~\ref{reducedclosesubchemedim2}). Geometrically, if $X$ is an affine $k$-scheme of finite type and $\rt (X)=1$, then the blowup of the affine ambient space along the defining ideal of $X$ is just a naive blowup that arises from the symmetric algebra. t is an interesting question whether the relation type of $X$ measures the complexity of the blowup.    In Proposition~\ref{a.c.i} we characterize equidimensional strict almost complete intersection closed subscheme of dimension one in $\AA_k^n$ of relation type one. We show that there do not exist a finite set of points in the projective plane and an irreducible affine space curve of relation type $2$ (see Example~\ref{spacecurve} and Proposition~\ref{points}). 

Similarly, we can define the notion of relation type of a complex space germ,  a complex analytic variety, a formal germ, and an algebroid $k$-varieties( see Definition~\ref{spacegerms} and Definition~\ref{algebroid}). Finally, Proposition~\ref{analyticinvariant} shows that the relation type is an analytic and formal invariant of a scheme of finite type over $\mathbb{C}$. 
 \section*{acknowledgment}
The authors are heartily indebted to the late Ernest Kunz for his kind assessment regarding parts of the paper. The authors would like to thank  Francesc Planas-Vilanova for helpful discussions on the subject of the paper. The first author thanks the Iranian National Elites Foundation for the financial supports and the Institute for Advanced Studies in Basic Sciences for the hospitality during this study.  
\section{the andr\'e-quillen homology and the relation type}
Let $\phi\colon A\rightarrow B$ be a morphism of rings and let $N$ be a $B$-module. Assume that $\mathcal{L}_{\phi}$ is the cotangent complex of $\phi$. The $i$-th \textit{Andr\'{e}-Quillen homology} of $B$ over $A$ with coefficients in $N$, denoted by $\Ho_i(A,B,N)$, is the $i$-th homology module of $\mathcal{L}_{\phi}\otimes_BN$, that is uniquely defined in the derived category of the category of $B$-modules. Moreover, $\Ho_0(A,B,N) \cong \Omega_{B/A}\otimes_BN$ where  $\Omega_{B/A}$  is the  module of K{\"a}hler differentials of the $A$-algebra $B$. If $\phi$ is surjective, then $\Ho_1(A,B,N)\cong  I/I^2\otimes_B N$ where $I=\ker\phi$. 

Let $A\xrightarrow{\phi}B\xrightarrow{\psi} C$   be a diagram of rings and  homomorphism of rings.  There is an exact sequence of  triangle $$\mathcal{L}_{\phi}\otimes_BC\rightarrow\mathcal{L}_{\psi\phi}\rightarrow\mathcal{L}_{\psi}\rightarrow\textstyle\sum^1\mathcal{L}_{\phi}\otimes_BC, $$ in the derived category of the category of $C$-modules, where  $\textstyle\sum$ denotes the suspension functor.  The long exact sequence 
\[\begin{array}{ccccccccccc}
\cdots&\kern-9pt\rar&\kern-8pt\Ho_{n+1}(B,C,N)&\kern-8pt\rar\kern-8pt&\Ho_n(A,B,N )&\kern-8pt\rar\kern-8pt&\Ho_n(A,C,N )&\kern-8pt\rar\kern-8pt&\Ho_n(B,C,N )&\kern-8pt\rar\kern-8pt&\cdots\\
&&&\cdots\rar\kern-8pt& \Omega_{B\slash A}\otimes_BN&\kern-9pt\rar&\kern-11pt\Omega_{C\slash A}\otimes_CN&\kern-13pt\rar&\kern-11pt\Omega_{C\slash B}\otimes_CN&\kern-13pt\lar\kern-8pt&0,
\end{array}  \]
of  Andr\'{e}-Quillen homologies associated to the above triangle is called the \textit{Jacobi-Zariski} exact sequence. For more details we refer to~\cite{andre,quillen1,quillen2}.

Based on the work of F. Planas-Vilanova in~\cite{FV1} there is a relationship between the relation type of a standard graded algebra and  Andr\'{e}-Quillen homology. Quite generally, we refine upon this approach for the standard graded algebras.    

\claim{}\label{zeroAQ}\rm 
Let $A$ be a standard graded $R$-algebra. Let $A_+$ stand for the irrelevant ideal of $A$. The Jacobi-Zariski exact sequence associated to  $R\rar A\surjects R\cong A/A_+$ gives the isomorphism
\begin{equation}\label{formula1}
\Ho_{i}(A,R,R)\cong \Ho_{i-1}(R,A,R).
\end{equation}
In particular, we have the isomorphism of $R$-modules
\begin{equation}\label{formula2}
A_1\cong A_+/A_+^2\cong \Ho_{1}(A,R,R)\cong \Ho_{0}(R,A,R)\cong \Omega_{A/R}\otimes_A R\cong \Omega_{A/R}/A_+\Omega_{A/R}.
\end{equation}
Therefore, $$\Ho_0(R,\sym_R(A_1),R)\cong\Ho_0(R,A,R)\cong A_1.$$ 
\claim{}\label{firstAQ}\rm 
Let $A$ and $B$ be two standard graded $R$-algebras and $\phi\colon A\surjects B$ be a surjective graded $R$-algebra homomorphism.  Setting  $\mathfrak{a}\coloneqq\ker(\phi)$.
One has 
\[\Ho_1(A,B,R)\cong \mathfrak{a}/\mathfrak{a}^2\otimes_B R\cong \mathfrak{a}/A_+\mathfrak{a}.\]
We can consider $\Ho_1(A,B,R)$ as a graded $A$-module 
\begin{equation}\label{formula3}
\Ho_1(A,B,R)=\bigoplus_{n\geq 1}\Ho(A,B,R)_n=\bigoplus_{n\geq 1} (\mathfrak{a}/A_+\mathfrak{a})_n=\bigoplus_{n\geq 1}\mathfrak{a}_n/A_1\mathfrak{a}_{n-1}.
\end{equation}
\claim{}\label{firstAQPresentation}\rm 
Let $A'\stackrel{\phi}{\surjects} A\surjects B$ be a diagram of surjective graded morphisms  of standard graded $R$-algebras such that $A'=\sym_R(A'_1)$ and $A=\sym_R(A_1)$.  By  the Jacobi-Zariski exact sequence, we have  the following graded exact sequence
\[\Ho_1(A',A,R) \rar \Ho_1(A',B,R) \rar \Ho_1(A,B,R) \rar0.\]
Formula~\eqref{formula3} together with  $\ker\phi=(\ker\phi)_1A'$ implies that $\Ho_1(A',A,R)_n=0$ for all $n\geq 2$. We conclude that   $\Ho_1(A,B,R)_n\cong\Ho_n(A',B,R)$ for all $n\geq 2$. 
In particular,
$$\Q_n/(R[T_1,\ldots,T_n])_1\Q_{n-1}\cong \Jc_n/\sym_R^1(A_1)\Jc_{n-1},$$
where $\Q$ and $\Jc$ are the kernels of  the polynomial and the canonical symmetric presentations of a standard graded $R$-algebra $A$. 
\claim{}\label{secondAQ}\rm
Let $A$ be a standard graded $R$-algebra. By~\cite[Proposition 15.12]{andre}, there exists an exact sequence 
\begin{equation}\label{syzgetic}
0\rar \Ho_2(A,R,R) \rar H_1\otimes_A R\rar A^m\otimes_A R\rar A_+\otimes_A R\rar 0,
\end{equation}
where $H_1$ is the first Koszul homology module  on the generators $f_1,\ldots, f_m$ of $A_+$. On the other hand, by the results of A. Simis~\cite{Aron}(see also~\cite{Aron-Wolmer}), we obtain that
\[\delta(A_+):=\ker(H_1\rar A^m\otimes_A R)=\ker(\sym_A^2(A_+)\rar A_+^2), \] 
which is an invariant of $A_+$. Therefore, the exact sequence~\eqref{syzgetic} implies that 
\[\Ho_2(A,R,R)\cong \delta(A_+). \]
Another proof of the latter, using the five terms of a spectral sequence,  is given by D. Quillen in~\cite[Chapter II, 6.10]{quillen1}. 
If $A=\sym_R(A_1)$, then by~\cite[Example 2.3]{H.S.V}  $\sym(A_1)_+^n\cong A_+^n$ for all $n\geq 1$ implying that  $\Ho_2(A,R,R)=0=\Ho_1(R,A,R)$. 
\claim{}\rm\label{VilanovaApproch} 
Let $A$ be a standard graded $R$-algebra. The Jacobi-Zariski exact sequence associated to $R\rar S:=\sym_R(A_1)\surjects A$ with coefficients in the $A$-module $R$ gives rise to 
\[\cdots\rar \Ho_1(R,S,R)\rar \Ho_1(R,A,R)\rar \Ho_1(S,A,R)\rar \Ho_0(R,S,R)\rar \cdots.  \]
By~\eqref{secondAQ}, one has $\Ho_1(R,S,R)=0$ also  by~\eqref{zeroAQ}, $\Ho_0(R,S,R)=\Ho_0(R,A,R)$. Using the above exact sequence we get the graded isomorphism of $R$-algebras 
\[\Ho_1(R,A,R)\cong \Ho_1(S,A,R)\cong\bigoplus_{n\geq 2}\Jc_n/S_1\Jc_{n-1}\] 

\claim{}\label{secondAQPresentation}\rm 
Let $A'$ and $A$ be two standard graded $R$-algebras and $\phi\colon A'\surjects A$ be a surjective graded $R$-algebra homomorphism. By the Jacobi-Zariski exact sequence associated to the ring homomorphisms 
$A'\surjects A\surjects R\cong A'/A'_+\cong A/A_+$, one gets the exact sequence 
\begin{equation}
\Ho_2(A',R,R)\rar\Ho_2(A,R,R)\rar\Ho_1(A',A,R)\rar\Ho_1(A',R,R)\rar\Ho_1(A,R,R)\rar 0 
\end{equation}
If $\phi$ is the canonical  symmetric presentation of $A$, then by~\eqref{secondAQ}  $\Ho_2(A',R,R)=0$.  
By~\eqref{zeroAQ}, \eqref{firstAQ} and~\eqref{firstAQPresentation}, we have the exact sequence 
\begin{equation}\label{vanishinH2}
0\rar \Ho_2(A,R,R)\rar \bigoplus_{n\geq 1}\Jc_n/\sym_R^1(A_1)\Jc_{n-1}\rar A_1\rar A_1\rar 0,
\end{equation}
which shows that 
\[ \Ho_2(A,R,R)\cong\bigoplus_{n\geq 2}\Jc_n/\sym_R^1(A_1)\Jc_{n-1}.\]
Therefore, $A$ is  of the relation type  one if and only if $\Ho_2(A,R,R)=0$. 
\subsection{The relation type of standard graded algebras}
Motivating from the above discussion we define the relation type of a standard graded $R$-algebra in terms of the first and the second Andr\'{e}-Quillen homologies. 
\begin{Definition}\label{defRT}
	Let $A$ be a standard graded $R$-algebra.The relation type of $A$ is defined as follow\[\mathrm {rt}_R(A)\coloneqq \min \{ \ell \geq1 \ \colon \ \Ho_2(A,R,R)_n=\Ho_1(R,A,R)_n=0\ \mbox{for}\ n\geq \ell+1 \}.\]
\end{Definition}
\begin{Remark}
	Let $k$ be a field. The relation type of a standard $k$-algebra $A=k[x_1,\ldots, x_n]\slash I$  earlier defined by W.~V.~Vasconcelos, in~\cite[Definition 2.7]{vasconcelos1}, to be the least integer $d$ such that $I=(I_1,\ldots, I_d)$, where $I\subseteq(x_1,\ldots, x_n)^2$ is a homogeneous ideal of the polynomial ring $k[x_1,\ldots, x_n]$ and $I_i$ is the $i$-th graded component  of $I$. The definition~\ref{defRT} coincides with of that Vasconcelos. Let $A$ be a  Cohen-Macaulay standard graded $k$-algebra. By~\cite[Prposition  2.8]{vasconcelos1}, we have \[\mathrm{rt}_k(A)\leq \mathrm{r}(A)+1=\mathrm{reg}(A)+1,\] where $\mathrm{r}(A)$ is the reduction number of $A$ (see~\cite[Definition 2.1]{vasconcelos1}) and $\mathrm{reg}(A)$ is the Costelnuovo-Mumford regularity of $A$.
\end{Remark}
The following result collects the basic algebraic properties for the relation type of standard graded algebras. 
\begin{Proposition}\label{basechange}
	Let $A$ be a standard graded $R$-algebra. Then the following statements hold.
	\begin{enumerate}
		\item[\rm(a)] If $\phi\colon R\rightarrow S$ is a surjective ring homomorphism, then	$\mathrm{rt}_S(A\otimes_RS)\leq\mathrm{rt}_R(A)$;
		\item[\rm(b)] The relation type of $A$ is a local invariant, i.e,. \[ \mathrm{rt}_R(A)=\max \{ \mathrm{rt}_{R_{\fp}}(A_{\fp}) \ \colon \fp\in\spec R\}; \]
		In particular, \[ \mathrm{rt}_R(A)=\max \{ \mathrm{rt}_{R_{\fm}}(A_{\fm}) \ \colon \fm\in\mathrm{Max}( R)\}; \]
		\item[\rm(c)] If the ideal $\mathfrak{a}$ is  contained in the Jacobson radical of $R$, then
		\[\mathrm{rt}_{\widehat{R}}(\widehat{A})=\mathrm{rt}_R(A).\]
		Here, $\widehat{R}$ and $\widehat{A}$ are the $\fa$-adic completions of $R$ and $A$, respectively; 
		\item[\rm(d)] Let $A$ and $B$ be  two standard graded $R$-algebras.	If $\Tor_1^R(A,B)=0$, then
		\[
		\rt_R(A\otimes_RB)=\max\{\mathrm{rt}_R(A),\mathrm{rt}_R(B)\}.
		\]
	\end{enumerate}
\end{Proposition}
\begin{proof}
	$\rm(a)$  Let $\Q$ and $\tilde{\Q}$ be the kernels of polynomial presentations $R[T_1,\ldots, T_m]\surjects A$ and  $S[T_1,\ldots, T_m]\surjects A\otimes_RS$ of the standard graded $R$-algebra $A$ and the standard graded $S$-algebra $A\otimes_RS$, respectively. The natural homomorphism $\Q\surjects\Q\otimes_RS\surjects\tilde{\Q}$ of graded $R$-modules,  for $n\geq 2$ induces 
	$$\Q_n\slash R[T_1,\ldots, T_m]_1{\Q}_{n-1}\surjects{\tilde{\Q}}_n\slash S[T_1,\ldots, T_m]_1{\tilde{\Q}}_{n-1}.$$ Therefore there is an exact sequence $\Ho_1(R, A ,R)_n\rightarrow \Ho_1(S,A\otimes_RS,S)_n\rightarrow 0$ of $R$-modules which implies that $\mathrm{rt}_S(A\otimes_RS)\leq\mathrm{rt}_R(A)$.
	
	$\rm(b)$  By the base change property~\cite[6.2]{Iyen}, $\mathcal{L}_{i}\otimes_RR_{\fp}$ and $\mathcal{L}_{i_{\fp}}$ are homotopic, where $\mathcal{L}_{i}$ and $\mathcal{L}_{i_{\fp}}$ are the cotangent complexes of the ring homomorphisms $R\xrightarrow{i} A$ and $R_{\fp}\xrightarrow{ i_{\fp}}A_{\fp}$, respectively. Therefore,    the isomorphism of $R$-modules 
	\[\Ho_1(R,A,R)\otimes_RR_{\fp}\cong\Ho_1(R_{\fp},A_{\fp},R_{\fp})\] 
	 holds which implies the assertions.  
	
	$\rm(c)$ Note first that $\widehat{A}$ is a standard graded $\widehat{R}$-algebra. By a similar argument as in (b), we have
	$\Ho_1(R,A,R)\otimes_R\widehat{R}\cong \Ho_1(\widehat{R},\widehat{A},\widehat{R})$. Since $\widehat{R}$ is faithfully flat, one easily obtains  that $\mathrm{rt}_R(A)=\mathrm{rt}_{\widehat{R}}(\widehat{A})$.\\
	
	$\rm(d)$ It is easy to see that $A\otimes_RB$ is a standard graded $R$-algebra. Since $\Tor_1^R(A,B)=0$, one may apply ~\cite[Proposition 19.3]{andreSimplicial} to get
	\[\Ho_1(R,A\otimes_RB,R)\cong \Ho_1(R,A,R)\oplus\Ho_1(R,B,R).
	\]
	The claim is now immediately follows from the definition.   
\end{proof}

\begin{Example}
	Let $A$ be a cyclic standard graded $R$-algebra generated by $f\in A_1$. Consider $\phi$ to be the polynomial presentation of $A$. Then $\mathcal{Q}_n=(0:_Rf^n)T^n \subseteq R[T]$. One may easily obtain that  $\mathcal{Q}_{n+1}=(R[T])_1\mathcal{Q}_n$ if and only if $(0:_Rf^{n+1})=(0:_Rf^n)$. One has
	\[
	\mathrm{rt}_R(A)=\min\{n\geq 1\ \colon \  (0:_Rf^{n+1})=(0:_Rf^n) \}.
	\]
\end{Example}

\begin{Example}
	Let $S=R[x_1,\ldots, x_n]$ be a polynomial ring over a ring $R$. Denote  by $\mathrm{Mon}^d(S)$ the set of all monomials of degree $d$ in $S$. Note that $N=\mid \mathrm{Mon}^d(S)\mid=
	\binom{d+n-1}{n-1}$. Assume that $A=R[\mathrm{Mon}^d(S)]$ is the standard graded $R$-algebra generated by $ \mathrm{Mon}^d(S)$. Consider the polynomial presentation $\phi\colon R[T_1,\ldots, T_N]\rar A$ sending $T_i$ to the $i$-th monomial of degree $d$ in lexicographic order. Write $m_1,\ldots, m_r$ for the monomials of degree $d-1$ in $S$ in lexicographic order where $r=\binom{d+n-2}{n-1}$. Let $\mathbf{M}$ be a metric of size $n\times r$, whose $(i,j)$-th entry is the variable $T_k$ such that $\phi(T_k)=x_im_j$. By~\cite[Proposition 2.5]{jan-katz}, we have $\mathcal{Q}=I_2(\mathbf{M})$. Then $\mathrm{rt}_R(A)=2$. 
\end{Example}
We give the following example of a family of standard graded algebra over a field with an unbounded relation type. 
\begin{Example}
Let $k$ be a field and $S=k[x_1,\ldots,x_d]$. Consider the following monomials in $S$
\[ f_1:=x_1^2\prod_{i=3}^{d}x_i,\, f_2:=x_1^2\prod_{i=2}^{d-1}x_i,\, f_{j+1}:=x_1^3 \frac{ \prod_{i=2}^{d}x_i}{x_jx_{j+1}},\, f_{d+1}:=x_1\cdots x_d \quad\hbox{for}\ j=2,\ldots, d-1 .\]
 For $d$ even number, let $A_1$ be a standard graded $k$-algebra generated by $f_1,f_2, \ldots, f_d$ and for $d$ odd number let  $A_2$ be a standard graded $k$-algebra generated by $f_1,\dots,f_d,f_{d+1}$. Then 
$A_1\cong k[T_1,\ldots,T_d]/(F_1)$ and $A_2\cong k[T_1,\ldots,T_{d+1}]/(F_2)$, where 
\[F_1=T_{d-1}\prod_{i=1}^{d\slash 2-1}T_{2i}-T_d\prod_{i=1}^{d\slash 2-1}T_{2i-1},\, F_2= T_{d}\prod_{i=1}^{(d-1)\slash 2}T_{2i}-T_{d+1}\prod_{i=1}^{(d-1)\slash 2}T_{2i-1}. \]
Then  $\mathrm{rt}_k(A_1)=d\slash 2$ and $\rt_k(A_2)=(d+1)/2$.
\end{Example}
\subsection{The relation type of an ideal}
Let $R$ be a commutative Noetherian ring and $I$ be an ideal of $R$. The relation type of $I$ is defined to be $\mathrm {rt}_R(I)\coloneqq\mathrm {rt}_R(\Rees_R(I))$. An ideal of relation type one is called of linear type (see for example~\cite{H.S.V}).

The surjective $R$-algebra homomorphism $\alpha\colon\sym_R(I)\surjects \Rees_R(I)$ induces the surjective $R/I$-algebra homomorphism $\beta\colon \sym_{R/I}(I/I^2)\surjects \gr_{I}(R)$, where $\gr_{I}(R)=\Rees_R(I)\otimes_R R/I\simeq \Rees_R(I)/I\Rees_R(I)$ is the associated graded algebra of $I$. By~\cite[Theorem 1.3]{Valla} (see also a geometric proof~\cite[Theorem V1.3]{Keel}) $\alpha$ is an isomorphism if and only if $\beta$ is an isomorphism. 
In terms of the relation type we have that $\mathrm{rt}_R(I)=1$ if and only if $\mathrm{rt}_{R/I}(\gr_{I}(R))=1$ which is generalized by 
F. Planas-Vilanova~\cite[Proposition 3.3]{FV1} and W. Heinzer, M.K. Kim and B. Ulrich~\cite[Discussion]{HKU}. Namely, for an ideal $I\subset R$ the relation type of the Rees algebra and the associate graded ring of $I$ are equal.  
We reprove this fact  by using the downgrading homomorphism 
\[\lambda\colon \sym_R^{n+1}(I)\lar \sym_R^n(I),\quad\lambda(b_1.b_2\cdots b_{n+1})=b_1(b_2.b_3\cdots b_{n+1})\in\sym_R^n(I),  \] 
introduced by J. Herzog, A. Simis and W. Vasconcelos~\cite[Page 471]{H.S.V}. Note that $\lambda$ is well-defined and  it dose not make a difference which element is pulled out of a monomial product.  

\begin{Proposition}\label{reproved}
	Let $R$ be a Noetherian ring and $I$ be an ideal of $R$. Then 
	\[\mathrm{rt}_R(\Rees_R(I))=\mathrm{rt}_{R/I}(\gr_{I}(R)). \]
\end{Proposition}
\begin{proof}
	By Lemma~\eqref{basechange}(a), we have $\mathrm{rt}_{R/I}(\gr_{I}(R))\leq \mathrm{rt}_R(\Rees_{R}(I))$. It is enough to show that the other side of  the inequality holds. Let $n\geq 1$. Assume \[\lambda\colon\sym_{R}^{n+1}(I)\rar\sym_{R}^{n}(I)\]  to be  the  downgrading $R$-homomorphism as it defined in the above discussion. Let $\mathcal{J}$ be the kernel of the canonical presentation $\alpha\colon\sym_R(I)\surjects \Rees_R(I)$. It follows that $\lambda(\mathcal{J}_{n+1})\subseteq  \mathcal{J}_n $ and since $\lambda$ does not make a difference which element is pulled out of the monomial product, one has $\lambda(\sym^1_R(I).\mathcal{J}_n)\subseteq\sym^1_R(I).\mathcal{J}_{n-1}$, implying that we  have the following exact sequence of $R$-modules.
	\[
	\mathcal{J}_{n+1}\slash \sym^1_R(I).\mathcal{J}_n\xrightarrow{\bar{\lambda}}\mathcal{J}_n\slash \sym^1_R(I).\mathcal{J}_{n-1}\rar\Coker\bar{\lambda}\rar 0.
	\]
	Furthermore, since $\lambda(\mathcal{J}_{n+1})=I\sym^n_R(I)\cap\mathcal{J}_n$ then 
	\[
	\Coker(\bar{\lambda}_n)=\mathcal{J}_n\slash I\sym_R^n(I)\cap\mathcal{J}_n+\sym_R^1(I).\mathcal{J}_{n-1}.
	\] 
	Consider the commutative diagram
	\[
	\begin{tikzcd}
	0 \arrow[r] & \mathcal{J}_n \arrow[d, "
	"] \arrow[r, " "] & \sym_R^n(I) \arrow[d, ""] \arrow[r, "\alpha_n"] & \Rees_{R}(I)_n \arrow[d, ""] \arrow[r] & 0 \\
	0 \arrow[r] & \mathcal{L}_n \arrow[r, " "] &\sym_{R\slash I}^n(I\slash I^2)\arrow[r, "\beta_n"] & \gr_{I}(R)_n \ar[r] & 0\\
	\end{tikzcd}
	\]
	where $\beta$ stands for the canonical presentation. Then  $\mathcal{L}_n=\mathcal{J}_n\slash\lambda(\mathcal{J}_{n+1})$. It is now easy to see that the natural homomorphism $\phi_n\colon \sym_{R\slash I}^1(I\slash I^2).\mathcal{L}_{n-1}\rar\mathcal{L}_n$ is well defined and $\Im\phi_n=\sym_R^1(I).\mathcal{J}_{n-1}+\lambda({\mathcal{J}_{n+1}})\slash\lambda({\mathcal{J}_{n+1}})$. Hence 
	\[
	\mathcal{L}_n\slash \sym_{R\slash I}^1(I\slash I^2).\mathcal{L}_{n-1} =\mathcal{L}_n\slash\Im\phi_n\cong\mathcal{J}_n\slash I\sym_R^n(I)\cap\mathcal{J}_n+\sym_R^1(I).\mathcal{J}_{n-1}= \Coker\bar{\lambda}_{n}.
	\]
	Therefore, 
	\[
	\Ho_1(R,\Rees_R(I), R)_{n+1}\xrightarrow{\bar{\lambda}}\Ho_1(R,\Rees_R(I),R)_n\rar\Ho_1(R\slash I,\gr_I(R),R\slash I)_n\rar 0
	\]
	is an exact sequence of $R$-modules. It immediately follows that 
	$\mathrm{rt}_{R\slash I}(\gr_{I}(R))\geq\mathrm{rt}_R(\Rees_R(I))$, as wanted.
\end{proof}

\begin{Corollary}\label{rtA=rtA_+}
	Let $A$ be a standard graded $R$-algebra. Then
	\[
	\mathrm{rt}_A(A_+)=\mathrm{rt}_R(A),
	\]
	where $A_+$ stands for the irrelevant ideal of $A$.
\end{Corollary}
\begin{proof}
	The claim is an immediate consequence of Proposition~\ref{reproved}. Indeed,  
	\[	
	\mathrm{rt}_A(A_+)=\mathrm{rt}_{A\slash A_+}(\gr_{ A_+}(A/A_+))=\mathrm{rt}_R(R[A_+\slash A_+^2])=\mathrm{rt}_R(R[A_1])=\mathrm{rt}_R(A).
	\]
\end{proof}

\section{the invariance of the relation type }\label{invariant}
In this section, we deal with the following question is due to Jos\'e M. Giral (see~\cite{FV-RT}, Page 178). 
\begin{Question}\label{RTinvaraiant}
Is the relation type of an ideal an invariant of the quotient ring?
\end{Question}\label{Qinva}
For an ideal in a polynomial ring the above question is answered~\cite[Theorem 3.1]{FV-RT}. We investigate the invariance of the relation type for formal $k$-algebras,  analytic $k$-algebras ($k$ is a complete valued field), regular local rings and complete Noetherian local rings. We will also give a counter example to the above question (see Example~\ref{counterexample}) 
 
Let $k[[\mathbf{x}]]\coloneqq k[[x_1,\ldots,x_n]]$ be a ring of \textit{formal power series} in $n$ indeterminates $x_1,\ldots,x_n$ over a field $k$. The ring  $k[[\mathbf{x}]]$ is a Noetherian regular local ring with the maximal ideal $\fm:=(x_1,\ldots,x_n)k[[\mathbf{x}]]$ and  complete for the $\fm$-adic topology. Moreover, it is the completion of the polynomial ring $k[x_1,\ldots,x_n]$ for  $\fn:=\fm\cap k[x_1,\ldots,x_n]$-adic topology. Also,  it is the completion of $k[x_1,\ldots,x_n]_{\fn}$ for  $\fn k[x_1,\ldots,x_n]$-adic topology~\cite[Ch. VII, § 1]{ZSII}. A $k$-algebra  is called a \textit{formal $k$-algebra} if it is isomorphic to $k[[\mathbf{x}]]/I$ for some ideal $I\subseteq k[[\mathbf{x}]]$. A morphism of formal $k$-algebra is a morphism of $k$-algebra which is automatically local.  

Let $k$ be a complete valued field. We denote by $k\{\mathbf{x}\}\coloneqq k\{x_1,\ldots,x_n\}$ the subring of all \textit{convergent power series} in $k[[x_1,\ldots,x_n]]$. The ring $k\{\mathbf{x}\}$ is a Noetherian regular local ring with the maximal ideal $\fm=(x_1,\ldots,x_n)k\{\mathbf{x}\}$. The completion of $k\{\mathbf{x}\}$ with respect to the $\fm$-adic topology is $k[[\mathbf{x}]]$. A $k$-algebra is called an \textit{analytic $k$-algebra} if it is isomorphic to $k\{\mathbf{x}\}/I$ for some ideal $I\subseteq k\{\mathbf{x}\}$ .

In what follows, we work with the ring of formal power series $k[[\mathbf{x}]]$ over an arbitrary field $k$, and also with the ring of convergence power series $k\{\mathbf{x}\}$ where $k$ is a complete valued field. To avoid of repetition, we denote by $k\left<\mathbf{x}\right>$ either $k[[\mathbf{x}]]$ or  $k\{\mathbf{x}\}$, where in the latter case $k$ is a complete valued field. 

The following lemma that serves as our starting point is indeed a generalization of~\cite[Exercise 13, Chapter V, §5]{kunz} to the formal and convergence power series. 
\begin{Lemma}\label{k-iso}
	Let $R=k\left<x_1,\ldots,x_n\right>$ and $I$ be an ideal of $R$. Let  $S=k\left<y_1,\ldots,y_m\right>$ and $J$ be an ideal of $S$. Suppose that there exists an isomorphism of $k$-algebras $A:=R/I\cong S/J$. Then there is an isomorphism of $A$-algebras 
	$$\gr_{I}(R)\left[y_1,\ldots,y_m\right]\cong \gr_{J}(S)\left[x_1,\ldots,x_n\right].$$
\end{Lemma} 

\begin{proof}
	Let $\mathbf{x}\coloneqq x_1,\ldots,x_n$ and $\mathbf{y} \coloneqq y_1, \ldots, y_m$. Consider the ring  $T\coloneqq k\left<\mathbf{x},\mathbf{y}\right>$ and the ideal $\mathfrak{a}\coloneqq IT+(\mathbf{y})T$. Using the $k$-algebra homomorphism $$\phi\colon T\surjects R $$ with $\phi(f(\mathbf{x},\mathbf{y}))=f(\mathbf{x},0)$, we obtain that $\phi(\mathfrak{a})=I$. It is also straightforward to see that $\ker\phi=(\mathbf{y})T$. Since $y_1,\ldots,y_m$ is $T/IT$-regular, it follows that $(\mathbf{y})T\cap (IT)^n=(\mathbf{y})T(IT)^n\subseteq (\mathbf{y})T\mathfrak{a}^{n-1}$ for every $n\geq 1$. Hence $(\mathbf{y})T\cap \mathfrak{a}^n=(\mathbf{y})T\cap (IT)^n+(\mathbf{y})T (IT)^{n-1}\subseteq (\mathbf{y})T\mathfrak{a}^{n-1}$. By~\cite[Theorem 1.1]{VaVa}, $\ker(\varphi)^*=(y_1^*,\ldots,y_m^*)$ 	where $y_i^*$ is the residue class of $y$ in $\mathfrak{a}\slash \mathfrak{a}^2$. Now
	~\cite[Lemma 5.3]{kunz} shows that 
	
	\begin{equation*}
		\gr_\mathfrak{a}(T)\slash(y_1^*,\ldots,y_m^*)\cong\gr_I(R), 
	\end{equation*}
	Therefore,
	\begin{equation*}
		\gr_{\mathfrak{a}}(T)\slash(y_1^*,\ldots,y_m^*)\otimes_AA[\mathbf{y}]\cong\gr_I(R)\otimes_AA[\mathbf{y}].
	\end{equation*}
	Regarding that $\gr_{\mathfrak{a}}(T)$ is an $A[y_1,\ldots,y_m]$-module, we now have
	\begin{align*}
		\gr_{\mathfrak{a}}(T)\slash(y_1^*,\ldots,y_m^*)\otimes_AA[\mathbf{y}]&\cong\gr_{\mathfrak{a}}(T)\otimes_{A[\mathbf{y}]}A[\mathbf{y}]\slash(\mathbf{y})\otimes_AA[\mathbf{y}]\\&\cong\gr_{\mathfrak{a}}(T)\otimes_{A[\mathbf{y}]}A\otimes_{A}A[\mathbf{y}]\\&\cong\gr_{\mathfrak{a}}(T).
	\end{align*}
	Hence, we obtain the following isomorphism of $A$-algebras $$\gr_{\mathfrak{a}}(T)\cong\gr_I(R)[y_1,\ldots,y_m].$$ Using a similar argument as above for the ideal  $\mathfrak{b}\coloneqq JT+(\mathbf{x})T$ one has $\gr_{\mathfrak{b}}(T)\cong \gr_J(S)[x_1,\ldots,x_n]$.
	By the lifting Lemmas,~\cite[Lemma 1.14 and Lamma 1.27]{GLS}, we now have the following commutative diagram of $k$-algebras and $k$-homomorphisms
	$$
	\begin{tikzcd}[column sep=small]
k\left<\mathbf{x},\mathbf{y}\right> \arrow{r}{\bar{\phi}}  \arrow{d}{\pi} 
	&k\left<\mathbf{x},\mathbf{y}\right>\arrow{d}{\pi'} \\
	R\slash I\cong k\left<\mathbf{x},\mathbf{y}\right>\slash \mathfrak{a} \arrow{r}{{\phi}}& k\left<\mathbf{x},\mathbf{y}\right>\slash \mathfrak{b} \cong S\slash J
	\end{tikzcd}
	$$
	where $\phi$ and $\bar{\phi}$ are isomorphisms, and $\bar{\phi}(\mathfrak{a})=\mathfrak{b}$.  Apply ~\cite[Lemma 5.3]{kunz} for the isomorphism $\bar{\phi}$, we obtain that
	$ \gr_\mathfrak{a}(T)\cong\gr_{\mathfrak{b}}(T), $
	which completes the proof. 
\end{proof}
We now turn toward the main result. 
\begin{Theorem}\label{T1}
	Let $R=k\left<x_1,\ldots,x_n\right>$ and $I$ be an ideal of $R$. Let  $S=k\left<y_1,\ldots,y_m\right>$ and $J$ be an ideal of $S$.
	 Assume that there exists an isomorphism of $k$-algebras $R/I\cong S/J$. 	Then $\mathrm{rt}_R(I)=\mathrm{rt}_S(J)$. 
\end{Theorem}
\begin{proof} By Lemma~\ref{k-iso},  we have the following commutative diagram of $k$-algebras
	$$ \begin{tikzpicture}[-stealth,
	label/.style = { font=\footnotesize }]
	\matrix (m)
	[
	matrix of math nodes,
	row sep    = 2em,
	column sep = 2em
	]
	{
		R\slash I & \gr_I(R)[y_1,\ldots, y_m] & R\slash I  \\
		S\slash J & \gr_J(S)[x_1,\ldots, x_n] & S\slash J   \\
	};
	\foreach \i in {1,...,2} {
		\path
		let \n1 = { int(\i+1) } in
		(m-1-\i) edge node [above, label] {} (m-1-\n1)
		(m-2-\i) edge node [below, label] {} (m-2-\n1)
		(m-1-\i) edge node [left,  label] {$\cong$} (m-2-\i);
	}
	\path (m-1-3) edge node [left, label] {$\cong$} (m-2-3);
	\end{tikzpicture}
	$$
	such that each column is an isomorphism of $k$-algebras. Note that the commutativity of the second diagram can easily obtained from the proof of  Lemma~\ref{k-iso}. This implies that 
	$$\Ho_i(R\slash I,\gr_I(R)[y_1,\ldots,y_m], R\slash I)\cong \Ho_i(S\slash J,\gr_J(S)[x_1,\ldots,x_n], S\slash J),  $$
for all $i\geq 0$. 
The ring homomorphisms 
	 $
	R\slash I\rightarrow \gr_I(R)\rightarrow\gr_I(R)[y_1,\ldots,y_m]$ gives rise to the  Jacobi-Zariski exact sequence of
	Andr\'{e}-Quillen homologies 
	\begin{align*}
		\cdots&\rightarrow \Ho_{i+1}(\gr_I(R),\gr_I(R)[y_1,\ldots,y_m], R\slash I)
		\rightarrow \Ho_i(R\slash I,\gr_I(R),R\slash I)\rightarrow\\&\Ho_i(R\slash I,\gr_I(R)[y_1,\ldots,y_m], R\slash I)\rightarrow\Ho_{i}(\gr_I(R),\gr_I(R)[y_1,\ldots,y_m], R\slash I)\rightarrow\cdots.
	\end{align*}
Note also that    $$\Ho_i(\gr_I(R), \gr_I(R)[y_1,\ldots,y_m], R\slash I)=0 $$ for every $i\geq 1$. Therefore 	
$$\Ho_i(R\slash I,\gr_I(R),R\slash I)\cong \Ho_i(R\slash I,\gr_I(R)[y_1,\ldots,y_m], R\slash I). $$
Similarly, we get $$\Ho_i(S\slash J,\gr_J(S),S\slash J)\cong \Ho_i(S\slash J,\gr_J(S)[x_1,\ldots,x_m], R\slash J). $$
Thus
\[\Ho_i(R\slash I,\gr_I(R),R\slash I)\cong \Ho_i(S\slash J,\gr_J(S),S\slash J).  \] 
By the above isomorphism for $i=1$, Definition~\ref{defRT} and Proposition~\ref{reproved}, one has 
\[\rt_R(I)=\rt_{R/I}(\gr_I(R))=\rt_{S/J}(\gr_J(S))=\rt_S(J). \]
\end{proof}
\begin{Corollary}\label{cor-vilanova} Let $(R,\fm)$ and $(S,\fn)$  be  complete Noetherian regular  local rings both containing some fields. Assume furthermore  that $k\coloneqq R\slash \fm\cong S\slash \fn$. Suppose that there is an isomorphism of $k$-algebras 
	$R\slash I\cong S\slash J$ for two ideals $I$ and $J$ of $R$ and $S$, respectively. Then $\mathrm{rt}_R(I)=\mathrm{rt}_S(J)$.
\end{Corollary}
\begin{proof}
	Applying the Cohen structure theorem~\cite[Corollary 2, p.~206]{matsumura-alg}, we have $R\cong k[[x_1,\ldots,x_m]]$ and $S\cong k[[y_1,\ldots,y_n]]$ for some integer $m=\dim R$ and $n=\dim S$. The assertion now follows from Theorem~\ref{T1}.
\end{proof}
Using the following results~\cite[Theorem 3.1]{FV-RT} is now obvious.
\begin{Corollary}\label{cor-vilanova2} Let $(R,\fm)$ and $(S,\fn)$ be regular local rings both containing some fields. Assume further  that $k\coloneqq R\slash \fm\cong S\slash \fn$. Suppose that there is an isomorphism of $k$-algebras 
	$R\slash I\cong S\slash J$ for two ideals $I$ and $J$ of $R$ and $S$, respectively. Then $\mathrm{rt}_R(I)=\mathrm{rt}_S(J)$.
\end{Corollary}
\begin{proof}
	One obtains that the completions of two rings $R/I$ and $S/J$ are isomorphic with respect to $\fm/I$-adic  and $\fn/J$-adic filtrations, respectively. Therefore, $\mathrm{rt}_{\widehat{R}}(I\widehat{R})=\mathrm{rt}_{\widehat{S}}(J\widehat{S})$ by Corollary~\ref{cor-vilanova}. This implies that $\mathrm{rt}_R(I)=\mathrm{rt}_S(J)$ by Proposition~\ref{basechange}(c).
\end{proof}
\begin{Remark}\mbox{}\label{Nashrings}
\begin{enumerate}
\item A formal power series $f\in k[[\mathbf{x}]]$ is called  an algebraic power series if $f$ is algebraic over $k[\mathbf{x}]$.  The set consisting of all algebraic power series is a ring that is denoted by $k\kern-4pt\ll\kern-4pt\mathbf{x}\kern-4pt\gg$. Substituting the ring $R$ and $S$, respectively with  the algebraic power series rings $k\kern-4pt\ll\kern-4pt\mathbf{x}\kern-4pt\gg$ and $k\kern-4pt\ll\kern-4pt\mathbf{y}\kern-4pt\gg$, one may see that  Lemma~\ref{k-iso} and Theorem~\ref{T1} are still valid. Every homomorphic image of an algebraic power series ring $k\kern-4pt\ll\kern-4pt\mathbf{x}\kern-4pt\gg$ is called a \textit{Nash $k$-algebra}. It therefore follows from our discussion that the relation type is an invariant of Nash $k$-algebras.   For more on Nash rings see~\cite{Ruiz}.
\item It is well-known that the power series rings, convergence power series rings, algebraic power series rings and  complete regular local rings are examples of \textit{Henselian rings}. It would be interesting to investigate  Question~\ref*{RTinvaraiant} for the Henselian ring (Thanks to Christopher Chiu).  
\end{enumerate}
\end{Remark}
\begin{Definition}
Let $A=k[[\xx]]/I$  be a formal $k$-algebra. We define the \textit{relation type of $A$} as $\rt(A)=\rt_R(I)$. There is an analogous definition in the case of analytic $k$-algebras and  Nash $k$-algebras. By Theorem~\ref{T1} and  Corollary~\ref{cor-vilanova}, the notion of relation type  is well-defined and it is an invariant in the categories of  formal $k$-algebras, analytic $k$-algebras, complete Noetherian local ring containing a field with the same residue field and  regular local rings. 
\end{Definition}

The following example indicates that  if $R\slash I\cong S\slash J$ and either or both of the rings $R$ and $S$ is non regular ring, then one can not deduce that
$\rt_R(I)=\mathrm{rt}_S(J)$.
\begin{Example}\label{counterexample}
Let $k$ be an algebraically closed field and let $X$ be a non regular integral scheme of finite type over $k$. Hence the regular locus $X_{\rm reg}$ is a non empty open set. Since the set of closed points of $X$ is very dense by~\cite[Proposition 3.33]{Gortz-wed},  there exist $x,y\in X$ such that $x$ is a regular point and $y$ is a non regular point. We also have 
	$$ \mathcal{O}_{X,x}/{\fm_x}\cong\mathcal{O}_{X,y}/{\fm_y}\cong k.$$	Hence $\mathrm{rt}_{\mathcal{O}_{X,x}}(\fm_x)=1$ while  $\mathrm{rt}_{\mathcal{O}_{Y,y}}(\fm_y)>1$ as $y$ is a singular point. 
\end{Example}
We close this section with the following remark. 
\begin{Remark}
Let $A=k\left<x_1,\ldots,x_n\right>\slash I$ be either a formal $k$-algebra or an analytic $k$-algebra. When the relation type of $A$ is at least $2$, the equidimensional components of $I$ of dimension $1$ and $n$ are irrelevant with respect to the relation type. The proof of this statement is analogous to the polynomial case in~\cite[Proposition 4.4]{FV-RT}.    
\end{Remark}
	
\section{the relation type of algebraic, analytic and algebroid varieties}
This section deals with the notion of relation type for $k$-schemes of finite type, analytic and algebroid varieties. 
\subsection{Schemes of finite type over a field}
Let  $X$ be an affine scheme of finite type over a field $k$. We define \textit{the relation type of $X$} as $\mathrm{rt}(X)=\mathrm{rt}(\mathcal{O}_X(X))$ where $\mathcal{O}_X(X)$ is the coordinate ring or the global section of $X$. Since the category of affine schemes of finite type over a field $k$ is anti-equivalent to the category of finitely generated $k$-algebras, by Corollary~\ref{cor-vilanova2}, one can deduce that the relation type of $X$ is well-defined and it is an invariant of affine schemes of finite type over a field $k$. In particular, the relation type is an invariant of affine varieties. 

Let $X$ be a scheme of finite type over a field $k$, i.e., a scheme locally of finite type over $k$ and quasi-compact. By using the local invariant property of the relation type, we define the relation type of $X$ as follow
$$\rt(X)=\max \{\rt(\OO_{X,x}) \ \colon x\in X \}.$$ 
Therefore, the definition of the relation type for $k$-schemes of finite type is well-defined and it is an invariant. 
\begin{Proposition}\label{regular}
	Let $X$ be a regular $k$-scheme of finite type. Then $\rt(X)=1$. 
\end{Proposition} 
\begin{proof}
	We may assume that $X=\spec{k[x_1,\ldots,x_n]/I}$ is an affine scheme of finite type over $k$. We have 
	\[\mathrm{rt}(X)=\max \{ \mathrm{rt} (\mathcal{O}_{X,x})\ \colon \ x\in X \}. \] 
	If $x\in X$ is a regular point, then  $\mathcal{O}_{X,x}\cong k[x_1,\ldots,x_n]_{p_x}/I_{p_x}$ which is a regular local ring. Therefore, by~\cite[Proposition 2.2.4]{Herzog.B} the ideal $I_{p_x}$ is generated by a regular sequence. This implies that $\mathrm{rt}(\mathcal{O}_{X,x})=\mathrm{rt}(I_{p_x})=1$. It turns out that 
	\[\mathrm{rt}(X)=\max \{ \mathrm{rt} (\mathcal{O}_{X,x})\  \colon\  x\in\Sing (X) \}. \]	
	In particular, If $X$ is  regular, then $\mathrm{rt}(X)=1$. 
\end{proof}

\begin{Example}\label{Zeroschemes} Let $X$ be a zero dimensional reduced affine scheme of finite type over a field $k$. Then $X$ consists of finitely many   reduced points and for each point $x\in X$, the local ring $\mathcal{O}_{X,x}$ is regular. Thus $\mathrm{rt}(X)=1$.
\end{Example}
\begin{Example}\label{reducedclosesubchemedim2}
	Let $X$ be a reduced closed subscheme of the $2$-dimensional affine space $\mathbb{A}_k^2=\spec{k[x,y]}$ over a field $k$. Then $X=V(I)\subseteq \mathbb{A}_k^2$ defined by a radical ideal $I\subseteq R=k[x,y]$. We have 
	\[\mathrm{rt}(X)=\mathrm{rt}(R/I)=\mathrm{rt}_R(I)=\max\{ \mathrm{rt}_{R_{\fm}}(I_{\fm})\ \colon\ \fm \ \mbox {is maximal ideal}, \  I\subseteq \fm \}.\]
	Since $I$ is radical and  in $R$ all non-maximal prime ideals of $R$ are principal, it follows that  $I=(f)\cap \fm_1\cap \cdots\cap\fm_s$ with $(f)\nsubseteq \fm_i$, where $(f)=(f_1)\cap\cdots\cap (f_m)$ is a prime ideal  and $\fm_i$ are maximal ideals. It turns out that $I_{\fm}=\fm R_{\fm}$ and hence $\mathrm{rt}_{R_{\fm}}(I_{\fm})=1$ as $\fm$ is a complete intersection. Therefore, $\mathrm{rt}(X)=1$. In particular, the blowup of affine space $\mathbb{A}_k^2$ along a reduced closed subscheme is just a naive blowup that arises from the symmetric algebra.  
\end{Example}
An ideal $I$ in a ring $R$ is called \textit{strict almost complete intersection} if $I$ is minimally generated by $h+1$ elements where $h$ is the height of $I$.  
\begin{Proposition}\label{a.c.i}
	Let $X=V(I)$ be an equidimensional closed subscheme of dimension one in $\AA_k^n$ defined by a strict almost complete intersection ideal $I\subseteq R=k[x_1,\ldots,x_n]$. Then $\rt(X)=1$ if and only if $\OO_{X,x}$ is a complete intersection for any maximal point $x\in X$. 
\end{Proposition}
\begin{proof}
	Let  $\rt(X)=1$. Then $\rt(\OO_{X,x})=\rt_{R_{\fp_x}}(I_{\fp_x})=1$ for any point $x\in X$. Assume that $x\in X$ is a maximal point which correspond to the minimal prime ideal $\fp_x$ of $I$.   By assumption and~\cite[Proposition 2.4]{H.S.V}, one has
	\[n-1\leq  \grade(I_{\fp_x})\leq\hht(I_{\fp_x})\leq \mu(I_{\fp_x})\leq \hht(\fp_x)=n-1.  \]
	where $\grade(-)$, $\hht(-)$ and $\mu(-)$ are the length of the longest regular sequence of elements in the ideal, height of the ideal and the minimal number of generators of the ideal, respectively. Hence $\hht(I_{\fp_x})=\mu(I_{\fp_x})=n-1$, which shows that $\OO_{X,x}$ is a complete intersection at maximal points of $X$. 
	
	For the inverse, it is enough to show that $\rt_{R_{\fp_x}}(I_{\fp_x})=1$ for all closed points of $x\in X$. Since the prime ideals containing $I$ are maximal ideals and its minimal primes, the latter all of height $n-1$, the assertion immediately follows from Theorem 9.1 in~\cite{H.S.V}. 
\end{proof}
\begin{Example}\label{spacecurve}
Let $X\subseteq \AA_k^3$ be an irreducible space curve. Denote by $\fp\subseteq R=k[x,y,z]$ the defining ideal of $X$ which is a codimension $2$ perfect ideal of projective dimension $1$. If $\fp$ is complete intersection or almost complete intersection, then $\rt_R(\fp)=1$. If $\fp$ is generated by at least four generators, then $\rt_R(\fp)\geq 2$. We claim that $\rt_R(\fp)>2$. If   $\rt_r(\fp)\leq 2$, since
$\rt_R(\fp)=\min \{ r\geq 1 \ | \ \Ho_{1}(R,\Rees_{R}(\fp),R )_n=0\ \hbox{for all }\ n\geq r+1  \}$, it follows that $\Ho_{1}(R,\Rees_{R}(\fp),R )_n=0$ for all $n\geq 3$. By~\cite[Remark 2.9]{H.S.V}, one has  $\Ho_{1}(R,\Rees_{R}(\fp),R )_2=0$ which implies that $\rt(\fp)=1$, a contradiction. Therefore, either $\rt_R(X)=1$ or else $\rt(X)=\mathrm{rt}_R(\fp)\geq 3$. In particular, there do not exist irreducible space curves in $\AA_k^3$ with relation type $2$. It would be interesting to characterize irreducible space curves $X$ of embedding dimension $3$ such that $\rt(X)=3$. 
\end{Example}
Let $X\subseteq \PP_k^n$ be a projective variety. We define the relation type of $X$ as the relation type of affine cone over $X$, that is $\mathrm{rt}(X):=\mathrm{rt}(C(X))$, where $C(X)$ is the affine cone over $X$.
\begin{Proposition}\label{points}
	Let $X$ be a finite set of $s$ points in the projective plane $\PP_k^2$ over a field $k$.  Then either $\mathrm{rt}(X)=1$ or else $\mathrm{rt}(X)\geq 3$. 
\end{Proposition}
\begin{proof}
	Let $I(X)=\bigcap_{p\in \mathbb{X}} I(p)$ stand for the defining ideal of $X$ where $I(p)$ is a prime ideal generated by two linear forms in $R=k[x,y,z]$. Then the coordinate ring of $X$ is  $A=R/I(X)$ which is a reduced equidimensional $k$-algebra of dimension $1$. By definition, $\mathrm{rt}(X)=\mathrm{rt}(C(X))=\mathrm{rt}(A)=\mathrm{rt}_R(I(X))$.  If $I(X)$ is a complete intersection,  then $\mathrm{rt}(I(X))=1$. Assume that $I(X)$ is an almost a complete intersection. Since the localization of $I(X)$ at each prime ideal $I(p_i)$ for $i=1,\ldots, s$,  is a complete intersection, by Proposition~\ref{a.c.i}, one has   $\mathrm{rt}_R(I(X))=1$. If $I(X)$ is generated by at least four elements, then $I(X)$ is not of linear type~\cite[Proposition 2.4]{H.S.V}. Thus $\mathrm{rt}(I(X))> 2$. Indeed, if $\rt(I(X))\leq 2$, since
	$\rt_R(I(X))=\min \{ r\geq 1 \ | \ \Ho_{1}(R,\Rees_{R}(I(X)),R )_n=0\ \hbox{for all }\ n\geq r+1  \}$, it follows that $\Ho_{1}(R,\Rees_{R}(I(X)),R )_n=0$ for all $n\geq 3$. By~\cite[Remark 2.9]{H.S.V}, one has  $\Ho_{1}(R,\Rees_{R}(I(X)),R )_2=0$ which implies that $\rt(I(X))=1$, a contradiction. Therefore, $\rt(X)=\mathrm{rt}_R(I(X))\geq 3$. 
\end{proof}
\begin{Example}
	Let $X\subseteq \PP_k^2$ be a set of six points in the projective plane over a field $k$ of characteristic zero such that three points are collinear and the remaining three points are in the general linear position. We show that $\rt(X)=3$. 
	
	 We may assume that the three points are in the line $z=0$ and  $p_1=[a_1:1:0],\, p_2=[a_2:1:0],\, p_3=[a_3:1:0]$ with $a_i\in k$ and  $a_i\neq 0,1,-1$ and by a projective transformation $p_4=[1:0:1],\, p_5=[0:1:1],\, p_6=[0:1:1]$. The ideal $I(X)$ is generated by the following  forms of degree $3$
	\[ xyz,\, x^2z-xz^2,\, y^2z-yz^2,\, x^3-e_1x^2y+e_2xy^2-e_3y^3+e_3yz^2-xz^2,\]
where $e_1:=a_1+a_2+a_3,\, e_2:=a_1a_2+a_1a_3+a_2a_3,\, e_3:=a_1a_2a_3$. The ideal $I(X)$ has the following syzygy matrix
\[
M:=\begin{bmatrix}
x&0&-e_3(y+z)\\
z-y&x-z&e_2y-e_1z\\
0&-y&x-e_1y+z\\
0&0&z
\end{bmatrix}.
\] 
The defining ideal $\Q_1\subseteq S=R[T_1,T_2,T_3,T_4]$ of the symmetric algebra of $I(X)$ is generated by the linear forms 
\[xT_1+(z-y)T_2,\, (x-z)T_2-yT_3,\,(-e_3(y+z))T_1+(e_2y-e_1z)T_2+(x-e_1y+z)T_3+zT_4. \]
Thus the Jacobian dual matrix( see~\cite{vasconcelos1}) $B(M)$ is of the form
\[ 
B(M):=\begin{bmatrix}
T_2&-T_3&-T_2\\
-T_3&e_2T_1-e_3T_2+e_1T_3&e_3T_1+e_1T_2-T_3+T_4\\
T_1&-T_2&T_2
\end{bmatrix}.
\]
Since the height of $\det(B(M))=4-3=1$, by~\cite[Corollary 8.2.9]{Vascon} one has $\Q=(\Q_1,\det(B(M)))$ which proves that $\rt(X)=3$.  
\end{Example}
In this regard, it is natural to pose the problem of 	\textit{characterizing configuration of points $X$ in the projective plane such that  $\rt(X)=3$}.  

\begin{Example}
	Let $X=\{p_1,\ldots,p_g,q_1,\ldots,q_g\}$ be a set of $2g$ distinct points in the projective line $\mathbb{P}_k^1$ over an algebraically closed field $k$ of characteristic zero with $g\geq 3$. By a projective transformation, we may assume that $p_i=[1:a_i]$ and $q_i=[1:b_i]$ where $a_i,b_i\in k$. We decompose $X=X_1\cup\ldots\cup X_g$, where $X_i=\{p_i,q_i\}$ for $i=1,\ldots,g$.  Denote by $Q_i=I(X_i) \subseteq R=k[x,y]$, the defining ideal of $X_i$ which is a principal ideal generated by homogeneous polynomial $Q_i=(a_ix-y)(b_ix-y)$. Setting 
	$f_i:=\prod_{j=1,i\neq j}^{g}Q_j.$
	
	Let 
	$\varphi: \mathbb{P}_k^1\dashrightarrow \mathbb{P}_k^{g-1}$ be a rational map defining by the polynomials $f_1,\ldots, f_g$. By construction $\varphi(p_i)=\varphi(q_i)$. The image of $\varphi$ is called \textit{canonical nodal curve }of genus $g\geq 3$.  For small $g$ it is possible to given an explicit geometric description of the canonical  nodal curves of genus $g$. Thus, for genus $3$ the canonical  nodal curve  is a quartic curve in $\mathbb{P}_k^2$, while for genus $4$ the canonical nodal curves is intersections of a quadric and a cubics in $\mathbb{P}_k^3$. 
	
	Let $I\subseteq R$ be an ideal generated by $f_1,\ldots,f_g$. Then the ideal $I$ is a Hilbert-Burch ideal with the minimal free resolution 
	\[0\lar R^{g-1}(-2g)\stackrel{M}\lar R^{g}(-(2g-2))\lar I\lar 0, \]
	where 
	\[M:=\begin{bmatrix}
	Q_1& 0&0&\cdots&0\\
	0&Q_2&0&\cdots&0\\
	0&0&Q_3&\cdots&0\\
	\vdots&\vdots&\vdots&\ddots&\vdots\\
	0&0&0&\cdots&Q_{g-1}\\
	-Q_g&-Q_g&-Q_g&\cdots&-Q_g
	\end{bmatrix}
	.\]
\end{Example}
Let $X=V(I)\subseteq \PP_k^{g-1}$. If $g=3$, then the defining ideal of the Rees algebra of $I$  generated by the defining ideal of the symmetric algebra of $I$, two moving conics and the implicit equation of degree $4$ which is a quartic curve in $\mathbb{P}_k^2$ with three nodes~\cite[Theorem 2.12]{DC}. If $g=4$, then the defining ideal of the Rees algebra of $I$  is generated by the defining ideal of the symmetric algebra of $I$, four moving conics and two implicit equations of degree $2$ and $3$~\cite[Example 5.2.2]{Hartshorne}. Therefore, 
\begin{equation*}
\rt(X)=\mathrm{rt}_R(I)=\left\{
\begin{array}{rl}
4 & \mbox{if}\ g=3\\
3 &  \mbox{if}\ g=4.
\end{array}
\right.
\end{equation*}
We conjecture that $\mathrm{rt}(X)=2$ for $g\geq 5$. In particular, the defining ideal of canonical nodal curve of genus $g\geq 5$ is a quadratic variety of degree $2g-2$. 
\begin{Remark}
For a given positive number $r\geq 1$. It would be interesting to classify affine and projective varieties of relation type $r$. What are the geometric properties of a finite set of points in the projective plane or an irreducible space curve of embedding dimension $3$ that the relation type is not $2$? Is there a geometric meaning to the relation type of affine and projective varieties?  
\end{Remark}
\subsection{Complex analytic spaces}
Let us recall the basic geometric objects in complex analytic geometry, namely complex space. We intend to give only basic definition that we need in the following. For more details see for example~\cite{GLS} and~\cite[Appendix by B. Moonen]{HKU}. 

Let $D\subseteq \mathbb{C}^n$ be an open subset. Denote by $\mathcal{O}_{\CC^n}$ the sheaf of holomorphic functions. Then the sheaf of holomorphic functions on $D$ is defined by $\OO_D=\imath^{-1}\OO_{\CC^n}$, where $\imath:D\hookrightarrow \CC^n$ is the canonical map. Let $\OO_{D,x}$ stand for the stalk of $\OO_D$ at the point $x\in D$. We  first need to recall a complex model space.

Let $D\subseteq \CC^n$ be an open subset and let  $\mathcal{I}\subseteq \OO_D$ be an ideal sheaf of finite type. Then $\OO_D/\mathcal{I}$ is a sheaf of rings on $D$.
Setting  
\[ V(\mathcal{I}):=\{ x\in D\  \colon\  \OO_{D,x}\neq \mathcal{I}_x \}=\{ x\in D\  \colon\  \left (\OO_{D}/\mathcal{I}\right )_x\neq 0 \}=\mathrm{supp}(\OO_{D}/\mathcal{I}).\]
 Recalling that $V(\mathcal{I})$  is an analytic set in $D$, for every $x\in D$ there an open neighborhood $U$ and  holomorphic functions $f_1,\ldots,f_m$   such that 
\[V(\mathcal{I})\cap U=V(f_1,\ldots,f_m)\cap U. \] 
A \textit{complex model space} is an analytic ringed space  $(X,\OO_X)$, where   $X:=V(\mathcal{I})$ and  $\OO_X:=\left(\OO_D/\mathcal{I}\right)|_X$. \color{black} A morphism of complex  model  space is just a morphism of analytic ringed spaces. 

An analytic ringed space $(X,\OO_X)$ is called a \textit{complex analytic  space} or a \textit{complex analytic variety}, if $X$ is a Hausdorff space and if every point of $X$ has an open neighborhood $U$ such  that the open analytic ringed subspace $(U,\OO_U)$ of $(X,\OO_X)$  is isomorphic to a complex model space. 

There is a connection between analytic $\CC$-algebras and complex analytic spaces. In fact, every stalk of the structure sheaf $\OO_X$ is an analytic $\CC$-algebra and conversely, every analytic $\CC$-algebra can be obtained as the stalk of a complex analytic space. More precisely, let $(X,\OO_X)$ be a complex analytic space and $x\in X$. If follows from the definition that there exist $f_1,\ldots,f_m\in \CC\{\xx\}$ such that $\OO_{X,x}\cong\OO_{\CC^n,\mathbf{0}}/\mathcal{I}_{\mathbf{0}}\cong \CC\{\xx\}/(f_1,\ldots,f_m)$. In this case, $\xx=(x_1,\ldots,x_n)$ and $f_1,\ldots,f_m$ is called \textit{local equation} and \textit{local coordinate} for $X$ at the point $x$, respectively. On the other hand, let   $A=\OO_{\CC^n,\mathbf{0}}/J$ be a given analytic $\CC$-algebra. Since $\OO_{\CC^n,\mathbf{0}}$ is Noetherian, there exists a 
neighborhood $U$ of $\mathbf{0}\in \CC^n$ and  holomorphic functions $f_1,\ldots,f_m\in \OO_U$, whose germs $(f_1)_x,\ldots,(f_m)_x$ generate $J$. Then $X:=V(f_1,\ldots,f_m)$ is a desired complex model space. 

A \textit{complex space germ} or \textit{singularity} is a tuple $(X,x)$ with $X$ a complex analytic space and $x\in X$. A morphism $f\colon (X,x)\rar (Y,y)$ of  complex space germs  is a morphism $f\colon X\rar Y$ of complex analytic spaces such that $f(x)=y$. The complex space germs with their morphism form a category which we denote by $\mathcal{C}$. If $(X,x)$ be a complex space germ and $U$ is any neighborhood of $x$, then $(U,x)=(X,x)$ up to  isomorphism in $\mathcal{C}$.  In this case we say  $U$ is a \textit{representative of the germ} $(X,x)$. The category of complex space germs $\mathcal{C}$ is anti-equivalent to the category of analytic $\CC$-algebras~\cite[Theorem 3.3.3]{Her}. 

\begin{Definition}\label{spacegerms} 
Let $(X,x)$ be a complex space germ. We define \textit{the relation type of $(X,x)$} as $\rt(X)=\rt(\OO_{X,x})$. By Theorem~\ref{T1}, it is well-defined and an invariant. 
\end{Definition}
Let $(X,\OO_X)$ be a complex analytic space. We can define the relation type of $X$ as $\rt(X):=\sup\{\rt(\OO_{X,x})\ : \ x\in X \}$ which is well-defined as $\OO_{X,x}$ is an analytic $\CC$-algebra and it is invariant of complex analytic spaces. 

Let $(X,\OO_X)$ be a complex analytic space and $x\in X$ a point.  The dimension of $X$ at the point $x$, $\dim_x X$, is defined by the Krull dimension of the local ring $\OO_{X,x}$. Also the embedding dimension of $X$ at the point $x$, $\mathrm{edim}_x X$,  is defined by the $\CC$-vector space dimension of $\fm_x/\fm^2_x$, where $\fm_x$ is the maximal ideal of $\OO_{X,x}$. We say $X$ is \textit{regular at point $x$}, if $\dim_x X=\mathrm{edim}_x X$. Otherwise, we say $X$ is\textit{ singular at the point $x$}. Let $\mathrm{Sing}(X)$ stand for the set of all singular point of $X$ which so-called \textit{singular space} or\textit{ singular locus} of $X$. By~\cite[Corollary 1.111]{GLS}, the singular space of $X$ is closed analytic set in $X$. 

\begin{Example}
Let $(X,\OO_X)$ be a complex analytic space. Then
\[\rt(X)=\sup\{\rt(\OO_{X,x}) \ \colon   \ x\in \mathrm{Sing}(X) \}. \] 
Therefore, any regular complex space $(X,\OO_X)$ is of the relation type $1$. 
\end{Example}
Let $X$ be a scheme of finite type over $\mathbb{C}$. We can associate to $X$ a complex analytic space $X_h$~\cite[Appendix B]{Hartshorne}. The following results prove  that the relation type is an analytic invariant of schemes of finite type over $\CC$. 
\begin{Proposition}\label{analyticinvariant}
Let $(X,\OO_X)$ be a scheme of finite type over $\mathbb{C}$. Then $\rt(X)=\rt(X_h)$. 
\end{Proposition}
\begin{proof}
For a point $x\in X$, one has $\OO_{X,x}\cong(\CC[\xx]/I)_{\fp_x}$ and $\OO_{{X_h},x}\cong \CC\{\xx\}/I\CC\{\xx \}$. The completion of both local rings with respect to their corresponding maximal ideals are isomorphic to $\CC[[\xx]]/I\CC[[\xx]]$.  Hence Lemma~\ref{basechange}(c) implies that $\rt(X)=\rt(X_h)$. 
\end{proof}

\subsection{Algebroid varieties and  formal germs}
Let $A$ be a complete local Noetherian ring with coefficient field $k$. Then $X=\spec{A}$ is called \textit{algebroid variety}. The local study of $k$-scheme of finite type can be reduced to the study of algebroid varieties. 

Let $X$ be a scheme locally of finite type over a filed $k$.  The local ring $\mathcal{O}_{X,x}$ of $X$ at the point $x\in X$ is equipped with the $\fm_x$-adic topology whose basis of neighborhoods  of $0$ is given by the powers $\fm_x^n$ of maximal ideal $\fm_x$ of $\mathcal{O}_{X,x}$. The induced completion $\widehat{\mathcal{O}}_{X,x}$ is called the \textit{complete local ring} of $X$ at $x$. The scheme defined by $\widehat{\mathcal{O}}_{X,x}$ is called a \textit{formal germs} of $X$ at the point $x$, ( or the \textit{algebroid variety} associated to $X$ at the point $x\in X$) denoted by $(\widehat{X},x)$.  

The natural ring homomorphism $\mathcal{O}_{X,x}\rar \widehat{\mathcal{O}}_{X,x}$ defines a morphism $(\widehat{X},x)\rar (X,x)$ in the category of schemes from formal germ to the germs of $X$ at the point $x$. 

A map $f\colon (\widehat{X},x)\rar (\widehat{Y},y)$ between two formal germs is defined as a local algebra homomorphism $\widehat{f}\colon \widehat{\mathcal{O}}_{Y,y}\rar \widehat{\mathcal{O}}_{X,x}$. It is also called  \textit{formal map} between $X$ and $Y$ at $x$. Two $k$-schemes locally of finite type $X$ and $Y$ are formally isomorphism at a point $x\in X$, respectively $y\in Y$, if the complete local rings $\widehat{\mathcal{O}}_{Y,y}$ and $\widehat{\mathcal{O}}_{X,x}$ are isomorphic as $k$-algebras. 

The formal germ $(\widehat{\mathbb{A}_k^n},x)$ of  the affine space of dimension $n$ at the point $x$ is given by the formal power series ring $\widehat{\OO}_{\mathbb{A}_k^n,x}\cong k[[x_1-a_1,\ldots,x_n-a_n]]$. If $X\subseteq \mathbb{A}_k^n$ be a closed subscheme defined by the ideal $I\subseteq k[x_1,\ldots,x_n]$, then the formal germ is  $$\widehat{\OO}_{X,x}=\widehat{\OO}_{\mathbb{A}_k^n,x}/\widehat{I}\cong k[[x_1-a_1,\ldots,x_n-a_n]]/\widehat{I},$$
where $\widehat{I}=I\widehat{\OO}_{\mathbb{A}_k^n,x}$ denotes the extension of $I$ to $\widehat{\OO}_{\mathbb{A}_k^n,x}$. 
\begin{Definition}\label{algebroid}
Let $X=\spec{A}$ be an algebroid variety. We  define the relation type of $X$ as $\mathrm{rt}(X)=\mathrm{rt}(A)$ which is well-defined and an invariant by Theorem~\ref{T1}. In particular, the relation type is an invariant of formal germs. 
\end{Definition}
Let $X$ be a $k$-scheme of finite type. Then 
\begin{eqnarray}
\nonumber \rt(X)=\max\{\rt(\widehat{\OO}_{X,x})\  \colon  \  x\in X \}. 
\end{eqnarray}
This shows that the relation type is a formal invariant. 
\begin{Example}
	The map $f\colon \mathbb{A}_k^1\rar \mathbb{A}_k^2$ given by $t\mapsto(t^2-1, t(t^2-1))$ is a regular morphism which induces a birational isomorphism onto the curve $X\subseteq \mathbb{A}_k^2$ defined by the polynomial $f=y^2-x^2-x^3$. The inverse $(x,y)\mapsto y/x$ is a rational map  on $X$ and regular on $X\setminus\{0\}$. The germ $(X,0 )$  is not isomorphic to the germ $(Y,0)$ of the union $Y$ of the two diagonals in $\mathbb{A}_k^2$ defined by the polynomial $g=x^2-y^2$. The formal germs $(\widehat{X},0)$ and $(\widehat{Y},0)$ are isomorphism via the map $\varphi\colon k[[x,y]]\rar k[[x,y]]$ given by $(x,y)\mapsto(x\sqrt{1+x},y)$. Let $I(f)=(f,\partial f/\partial x,\partial f/\partial y)=(f,x+x^2,y)$ and $I(g)=(g,x,y)$ be the Jacobian ideals of $f$ and $g$ in the ring  $R=k[x,y]$, respectively. By using the isomorphism $\varphi$, we have the $k$-algebra isomorphism  
	\[ k[[x,y]]/ I(f)k[[x,y]]\cong k[[x,y]]/I(g)k[[x,y]]. \]
	Thus 
	$$\rt_R(I(f))=\rt(R/I(f))=\rt(R/I(g))=\rt_R(I(g))=1.$$ 
	We aim  to pursue the relation type of the Jacobian ideal  of a hypersurface singularity in future work.  
\end{Example}
We close with the following remark. 
\begin{Remark}
Let $X$ be normal crossing $k$-scheme of finite type. Then for all point $x\in X$, the formal germ $(\widehat{X}, x)$ is isomorphic to the formal germ $(\widehat{Y}, \mathbf{0})$, where $Y$ is a scheme defined locally at $\mathbf{0}$ in $\AA_k^n$, up to isomorphism, by a monomial ideal in $R=k[x_1,\ldots,x_n]$. Therefore, 
\begin{eqnarray}
\nonumber \rt(X)&=&\max\{\rt(\widehat{X}, x)\ | \ x\in X \}\\
\nonumber &=& \max \{\rt_{\widehat{\mathcal{O}}_{\AA_k^n,\mathbf{0}}}(\widehat{I_x}) \ | \ x\in X  ,\, \  (\widehat{X},x)\cong (\widehat{V(I_x)},0) \},
\end{eqnarray}
where $I_x\subseteq R$ is a monomial ideal that corresponds to the normal crossing point $x\in X$. Therefore, the problem of characterizing the relation type of normal crossing $k$-scheme of finite type is reduced to characterizing the relation type of monomial ideals.
\end{Remark}

\end{document}